\input ./style/arxiv-vmsta.cfg
\documentclass[numbers,compress,v1.0.1]{vmsta}
\usepackage{vtexbibtags}
\usepackage{mathrsfs}
\usepackage[mathcal]{euscript}
\usepackage{stmaryrd}
\usepackage{mathbh}

\volume{7}% Updated by VTEXPTS2LaTeX.exe, 23.06.2020 11:00
\issue{2}% Updated by VTEXPTS2LaTeX.exe, 23.06.2020 11:00
\pubyear{2020}
\firstpage{135}% Updated by VTEXPTS2LaTeX.exe, 23.06.2020 11:00
\lastpage{156}% Updated by VTEXPTS2LaTeX.exe, 23.06.2020 11:00
\aid{VMSTA153}% Updated by VTEXPTS2LaTeX.exe, 15.05.2020 14:56
\doi{10.15559/20-VMSTA153}% Updated by VTEXPTS2LaTeX.exe, 15.05.2020 14:56
\articletype{research-article}

\DeclareMathDelimiter{\rlb}{\mathclose}{stmry}{"4A}{stmry}{"71}
\DeclareMathDelimiter{\lrb}{\mathopen}{stmry}{"4B}{stmry}{"79}
\DeclareMathOperator{\Var}{Var}

\startlocaldefs
\def\Eins{\mathbh{1}}
\newtheorem{thm}{Theorem}
\newtheorem{lemma}{Lemma}
\newtheorem{cor}{Corollary}
\newtheorem{prop}{Proposition}

\theoremstyle{definition}
\newtheorem{example}{Example}
\newtheorem{remark}{Remark}

\allowdisplaybreaks
\ifdefined\HCode % tex4ht
\newenvironment{defenum}{\begin{enumerate}}{\end{enumerate}}
\def\index#1{}
\else % LaTeX
\newenvironment{defenum}{\begin{description}}{\end{description}}
\fi
\endlocaldefs

\begin{document}

\begin{frontmatter}
\pretitle{Research Article}

\title{Single jump filtrations and local martingales}

\author[a,b]{\inits{A. A.}\fnms{Alexander A.}~\snm{Gushchin}\ead[label=e1]{gushchin@mi-ras.ru}\orcid{0000-0002-0020-7496}}
\address[a]{\institution{Steklov Mathematical Institute}, Gubkina 8, 119991 Moscow, \cny{Russia}}
\address[b]{\institution{National Research University Higher School of Economics},\break  Pokrovsky Boulevard 11, 109028 Moscow, \cny{Russia}}
%\thankstext[id=f1]{}

%\runtitle{}% auto-generated if not indicated
%\runauthor{}% auto-generated if not indicated

%\dedicated{}

\begin{abstract}
A single jump filtration
$({\mathscr{F}}_{t})_{t\in \mathbb{R}_{+}}$ generated by a random variable
$\gamma $ with values in $\overline{\mathbb{R}}_{+}$ on a probability
space $(\Omega ,{\mathscr{F}},\mathsf{P})$ is defined as follows: a set
$A\in {\mathscr{F}}$ belongs to ${\mathscr{F}}_{t}$ if
$A\cap \{\gamma >t\}$ is either $\varnothing $ or $\{\gamma >t\}$. A process $M$ is proved to be a local martingale with respect to this filtration
if and only if it has a representation
$M_{t} = F(t){\Eins }_{\{t<\gamma \}} + L{\Eins }_{\{t
\geqslant \gamma \}}$, where $F$ is a deterministic function and $L$ is
a random variable such that $\mathsf{E}|M_{t}|<\infty $ and
$\mathsf{E}(M_{t})=\mathsf{E}(M_{0})$ for every
$t\in \{t\in \mathbb{R}_{+}\colon {\mathsf{P}}(\gamma \geqslant t)>0\}$. This
result seems to be new even in a special case that has been studied in
the literature, namely, where ${\mathscr{F}}$ is the smallest
$\sigma $-field with respect to which $\gamma $ is measurable (and then
the filtration is the smallest one with respect to which $\gamma $ is a
stopping time). As a \xch{consequence,}{consequence.} a full description of all local
martingales is given and they are classified according to their global behaviour.
\end{abstract}

\received{\sday{24} \smonth{9} \syear{2019}}% Updated by VTEXPTS2LaTeX.exe, 15.05.2020 14:56
\revised{\sday{30} \smonth{4} \syear{2020}}% Updated by VTEXPTS2LaTeX.exe, 15.05.2020 14:56
\accepted{\sday{1} \smonth{5} \syear{2020}}% Updated by VTEXPTS2LaTeX.exe, 15.05.2020 14:56
\publishedonline{\sday{25} \smonth{5} \syear{2020}}

\end{frontmatter}
%

%s1 #&#
\section{Introduction}
\label{sec:1}

Starting with Dellacherie \cite{Dellacherie:70}, the following simple model
has been studied and intensively used in applications. Given a random variable
$\gamma $ with positive values on a probability space
$(\Omega ,{\mathscr{F}},\mathsf{P})$,\index{probability space} one considers the smallest filtration\index{filtration}
with respect to which $\gamma $ is a stopping time\index{stopping time} (or, equivalently, the
process ${\Eins }_{\{t\geqslant \gamma \}}$ is adapted). In particular,
Dellacherie gives a formula for the compensator\index{compensator} of this single jump process
${\Eins }_{\{t\geqslant \gamma \}}$.\index{single jump ! process} Chou and Meyer
\cite{ChouMeyer:1975} describe all local martingales\index{local martingale} with respect to this
filtration\index{filtration} and prove a martingale representation theorem. A significant
contribution is done in a recent paper by Herdegen and Herrmann
\cite{HerdegenHerrmann:16}, where a classification, whether a local martingale\index{local martingale}
in this model is a strict local martingale,\index{local martingale} or a uniformly integrable martingale,
etc., is given. Let us also mention some related papers
\cite{BoelVaraiyaWong,jacod1975multivariate,Jacod1976,Davis:1976,Elliott:1976,Neveu:1977,He:1983}, where, in particular,
local martingales\index{local martingale} with respect to the filtrations generated by jump processes
or measures of certain kind are studied.

Let us clarify that in the above model every local martingale\index{local martingale} has the form
%
%e1 #&#
\begin{equation}
\label{repr}
M_{t} = F(t){\Eins }_{\{t<\gamma \}} + H(\gamma ){\Eins }_{\{t
\geqslant \gamma \}},
\end{equation}
or
\begin{equation*}
M_{t} = F(t\wedge \gamma ) - K(\gamma ){\Eins }_{\{t\geqslant
\gamma \}},
\end{equation*}
where $\gamma $ is a random variable with values in, say,
$(0,+\infty )$, $F$, $H$, and $K=F-H$ are deterministic functions. Denote
by $G$ the distribution function of $\gamma $,
${\overline{G}}(t)=1-G(t)$, $t_{G}=\sup \{t\colon G(t)<1\}$ is the right
endpoint of the distribution of $\gamma $. Assume that
$\mathsf{E}|M_{t}|<\infty $, then
\begin{equation*}
{\mathsf{E}}(M_{t}) = F(t){\overline{G}}(t) + \int _{[0,t]}H(s)\,dG(s),
\end{equation*}
where the corresponding Lebesgue--Stieltjes integral is finite. If
$(M_{t})$ is a martingale,\index{martingale} then $\mathsf{E}(M_{t})=\mathsf{E}(M_{0})$, and this
equality can be written as
%
%e2 #&#
\begin{equation}
\label{e:fe}
F(t){\overline{G}}(t) + \int _{[0,t]}H(s)\,dG(s) = F(0)
\end{equation}
and can be viewed as a functional equation concerning one of functions
in $(F,G,H)$ or $(F,G,K)$, where other two functions are assumed to be
given. In fact, this equation takes place for $t<t_{G}$ or
$t\leqslant t_{G}$, the latter in the case where $t_{G}<\infty $ and
$\mathsf{P}(\gamma =t_{G})>0$. Moreover, it turns out that this is not only
the necessary condition but also the sufficient one for
$(M_{t})_{t\in \mathbb{R}_{+}}$ given by \eqref{repr} to be a local martingale.\index{local martingale}
This consideration allows us to reduce problems to solving this functional
equation. For example, to find the compensator\index{compensator} $F(t\wedge \gamma )$ of
${\Eins }_{\{t\geqslant \gamma \}}$ as in \cite{Dellacherie:70} one
needs to find a solution $F$ given $G$ and $K\equiv 1$. A possible way
to explain the idea in \cite{ChouMeyer:1975} is the following: The terminal
value $M_{\infty }$\index{terminal value} of any local martingale\index{local martingale} $M$ in this model is represented
as $H(\gamma )$, and to find a representation \eqref{repr} for $M$ it is
enough to solve the equation for $F$ given $G$ and $H$; the linear dependence
between $H$ and $F$ results in a representation theorem. Contrariwise,
in \cite{HerdegenHerrmann:16} the authors suggest to find $H$ from the
equation for given $F$ and $G$. This allows them to study global properties
of $M$.

In this paper we consider a more general model, where all randomness appears
``at time $\gamma $'' but it may contain much more information than
$\gamma $ does. We start with a random variable $\gamma $ on a probability
space $(\Omega ,{\mathscr{F}},\mathsf{P})$,\index{probability space} and define a single jump filtration
$({\mathscr{F}}_{t})$\index{single jump ! filtration} in such way that nothing happens strictly before
$\gamma $, $\gamma $ is a stopping time\index{stopping time} with respect to it, and the
$\sigma $-field ${\mathscr{F}}_{\gamma }$ of events that occur before or
at time $\gamma $ coincides with ${\mathscr{F}}$ (in fact, on the set
$\{\gamma <\infty \}$). We prove that every local martingale\index{local martingale} with respect
to this filtration\index{filtration} has the representation
%
%e3 #&#
\begin{equation}
\label{repr2}
M_{t} = F(t){\Eins }_{\{t<\gamma \}} + L{\Eins }_{\{t
\geqslant \gamma \}},
\end{equation}
where now $L$ is a random variable which is not necessarily a function
of $\gamma $. However, denoting $H(t)=\mathsf{E}[L|\gamma =t]$, we come to
the same functional equation of type \eqref{e:fe}.

Some results of the paper can be deduced from known results for marked
point processes, at least if ${\mathscr{F}}$ is countably generated; this applies, for
example, to Theorem~\ref{th:incr} about the compensator\index{compensator} of a single jump process.
Another example is Corollary~\ref{co:decomp} which says that every local
martingale\index{local martingale} is the sum of a local martingale\index{local martingale} of form \eqref{repr} and an
``orthogonal'' local martingale,\index{local martingale} the latter being characterised, essentially,
by the property $F(t)\equiv 0$. The reader can recognize in this decomposition
the representation of a local martingale\index{local martingale} as the sum of two stochastic integrals
with respect to random measures, see \cite{Jacod1976} and
\cite{Jacod1979}. However, our direct proofs are simpler due to the key
feature of our paper. Namely, we obtain a simple necessary and sufficient
condition for a process to be a local martingale\index{local martingale} and later exploit it.
A description of \emph{all\/} local martingales\index{local martingale} via a full description of
\emph{all\/} possible solutions to a functional equation of type \eqref{e:fe} is a simple consequence of this necessary and sufficient condition.
In particular, an absolute continuity type property of $F$ with respect
to $G$, considered as an assumption in \cite{HerdegenHerrmann:16}, is proved
to be a necessary condition. An elementary analysis of a functional equation
of type \eqref{e:fe} shows that, if $\gamma $ has no atom at its right
endpoint, there are different $F$ satisfying the equation for given
$H$ and $G$. In particular, there is a local martingale\index{local martingale} $M$ such that
$M_{0}=1$ and $M_{\infty }=0$; $M$ is necessarily a closed supermartingale.\index{supermartingale}\looseness=-1

Another important feature of our model, in contrast to Dellacherie's model,
is that it admits $\sigma $-martingales which are not local martingales.\index{local martingale}

Let us also mention some other papers where processes of form \eqref{repr} or \eqref{repr2} are considered. Processes of form \eqref{repr} with $t_{G}=\infty $ are typical in the modelling of credit
risk, see, e.g., \cite{JeanblancRutkowski:2000} and
\cite[Chapter~7]{JeanblancYorChesney:2009}, where usually $F$ is expressed
via $G$ and one needs to find $H$. Since $t_{G}=\infty $, such a process
is a martingale.\index{martingale} For example, in the simplest case
$F=1/{\overline{G}}$ and hence $H=0$. This process is the same that is mentioned
in two paragraphs above. Single jump filtrations\index{single jump ! filtration} and processes of form \eqref{repr2} appear in \cite{Gushchin:18} and \cite{Gushchin:20}. It is
interesting to note that, in \cite{Gushchin:20}, the random ``time''
$\gamma $ is, in fact, the global maximum of a random process, say, a convergent
continuous local martingale.\index{local martingale}

Section~\ref{sec:2} contains our main results. In Theorem~\ref{th:2} we
establish a necessary and sufficient condition for a process of type~\eqref{repr2}
to be a local martingale.\index{local martingale} This allows us to obtain a full description of
all local martingales\index{local martingale} through a functional equation of type~\eqref{e:fe}
in Theorem~\ref{th:solutions}. A similar description is available for
$\sigma $-martingales, see Theorem~\ref{th:sigmamart}. Finally, Theorem~\ref{th:class}
classifies local martingales\index{local martingale} in accordance with their global behaviour
up to $\infty $. Section~\ref{sec:4} contains the proofs of these results.
In Section~\ref{sec:5} we consider complementary questions. Namely, we
find the compensator\index{compensator} of a single jump process. We also consider submartingales\index{submartingale}
of class $(\Sigma )$, see \cite{Nikeghbali:06}, and show that their transformation
via a change of time leads to processes of type~\eqref{repr2}. As a consequence,
we reprove %\cite[Theorem 4.1]{Nikeghbali:06}.
Theorem 4.1 of \cite{Nikeghbali:06}.

We use the following notation: $\mathbb{R}_{+}=[0,+\infty )$,
$\overline{\mathbb{R}}_{+}=[0,+\infty ]$,
$a\wedge b = \min \,\{a,b\}$. The arrows $\uparrow $ and
$\downarrow $ indicate monotone convergence, while
$\lim _{s\upuparrows t}$ stands for $\lim _{s\to t,s<t}$.

A real-valued function $Z(t)$ defined at least for $t\in [0,s)$ is called
c\`{a}dl\`{a}g\index{c\`{a}dl\`{a}g} on $[0,s)$ if it is right-continuous at every $t\in [0,s)$ and
has a finite left-hand limit at every $t\in (0,s)$; it is not assumed that
it has a limit as $t\upuparrows s$. If, additionally, a finite limit
$\lim _{t\upuparrows s} Z(t)$ exists, then $Z(t)$ is called c\`{a}dl\`{a}g\index{c\`{a}dl\`{a}g}
on $[0,s]$. Functions $Z$ of finite variation\index{finite variation} on compact intervals are
understood as usually and are assumed to be c\`{a}dl\`{a}g.\index{c\`{a}dl\`{a}g} The variation at
$0$ includes $|Z(0)|$ as if $Z$ is extended by $0$ on negative axis. The
total variation of $Z$ over $[0,t]$ is denoted by $\Var (Z)_{t}$. We say
that $Z$ has a finite variation\index{finite variation} over $[0,s)$, $s\leqslant \infty $, if
$\lim _{t\upuparrows s} \Var (Z)_{t}<\infty $. We denote
$\Var (Z)_{\infty }:= \lim _{t\to \infty } \Var (Z)_{t}$.\looseness=-1

A filtration\index{filtration} on a probability space
$(\Omega ,{\mathscr{F}},\mathsf{P})$\index{probability space} is an increasing right-continuous family
$\mathbb{F}=({\mathscr{F}}_{t})_{t\in \mathbb{R}_{+}}$ of sub-$
\sigma $-fields of ${\mathscr{F}}$. No completeness assumption is made.
As usual, we define
${\mathscr{F}}_{\infty }=\sigma \bigl (\cup _{t\in \mathbb{R}_{+}}{
\mathscr{F}}_{t}\bigr )$ and, for a stopping time $\tau $\index{stopping time} the
$\sigma $-field ${\mathscr{F}}_{\tau }$ is defined by
\begin{equation*}
{\mathscr{F}}_{\tau }=\bigl \{A\in {\mathscr{F}}_{\infty }\colon A\cap
\{\tau \leqslant t\}\in {\mathscr{F}}_{t} \text{ for every }t
\geqslant 0\bigr \}.
\end{equation*}
A set $B\subset \Omega \times \mathbb{R}_{+}$ is \emph{evanescent\/} if
$B\subseteq A\times \mathbb{R}_{+}$, where $A\in {\mathscr{F}}$ and
$\mathsf{P}(A)=0$. We say that two stochastic processes $X$ and $Y$ are
\emph{indistinguishable\/} if $\{X\neq Y\}$ is an evanescent set.\index{evanescent set}

Since we do not suppose completeness of the filtration\index{filtration} $\mathbb{F}$, we
cannot expect that processes that we consider have all
\xch{paths c\`{a}dl\`{a}g}{c\`{a}dl\`{a}g paths}.\index{c\`{a}dl\`{a}g ! paths}
Instead we consider processes whose almost all paths are c\`{a}dl\`{a}g.\index{c\`{a}dl\`{a}g} Obviously,
for any c\`{a}dl\`{a}g process\index{c\`{a}dl\`{a}g ! process} $X$ adapted with respect to the completed filtration,
there is an a.s. c\`{a}dl\`{a}g $\mathbb{F}$-adapted process indistinguishable
from $X$. Furthermore, any $\mathbb{F}$-adapted process $X$ with a.s.
c\`{a}dl\`{a}g paths\index{c\`{a}dl\`{a}g ! paths} is indistinguishable from an $\mathbb{F}$-optional process
$Y$ whose paths are right-continuous everywhere and have finite left-hand limits
for $t<\rho (\omega )$ and $t>\rho (\omega )$, where $\rho $ is a
$\mathbb{F}$-stopping time with $\mathsf{P}(\rho <\infty )=0$; let us call
such $Y$ \emph{regular\/} and $\rho $ a
\emph{moment of irregularity\/} for $Y$. Dellacherie and Meyer
\cite[VI.5 (a), p.~70]{DellacherieMeyer:1982} prove that, if the filtration\index{filtration}
is not complete, every supermartingale $X$ (with right-continuous expectation)
has a modification $Y$ with the above regularity property. If we are given
just an adapted process $X$ with almost all
\xch{paths c\`{a}dl\`{a}g}{c\`{a}dl\`{a}g paths},\index{c\`{a}dl\`{a}g ! paths} we define
$\rho $ and $Y$ from values of $X$ on a countable set exactly as is done
in \cite{DellacherieMeyer:1982} in the case where $X$ is a supermartingale.\index{supermartingale}
Using \cite[Theorem IV.22, p. 94]{DellacherieMeyer:1978}, we obtain that
$\rho (\omega )=\infty $ and paths $X_{\cdot }(\omega )$ and
$Y_{\cdot }(\omega )$ coincide for those $\omega $ for which
$X_{\cdot }(\omega )$ is c\`{a}dl\`{a}g\index{c\`{a}dl\`{a}g} everywhere. Moreover, if
$\rho (\omega )<\infty $, then $Y_{t}(\omega )$ is c\`{a}dl\`{a}g\index{c\`{a}dl\`{a}g} for
$t<\rho (\omega )$ and one may put $Y_{t}(\omega )=0$ for
$t\geqslant \rho (\omega )$.

Processes with finite variation\index{finite variation} are adapted and not assumed to start from
$0$. A~moment of irregularity for them has additionally the property that
their paths have finite variation\index{finite variation} over $[0,t]$ for all
$t<\rho (\omega )$.

It is instructive to mention that, in our model, there is no need to use
general results on the existence of (a.s.) c\`{a}dl\`{a}g modifications for
martingales\index{c\`{a}dl\`{a}g ! modifications for martingales} since they can be proved directly. For example, if $L$ is an
integrable random variable with $\mathsf{E}L=0$, then the process $M$ given
by \eqref{repr2} with
$F(t)= \mathsf{E}[L|\gamma >t]{\Eins }_{\{t<t_{G}\}}$ satisfies
$M_{t}=\mathsf{E}[L|{\mathscr{F}}_{t}]$ a.s. for an arbitrary $t$. It is
trivial to check that this function $F$ has finite variation\index{finite variation} over any
$[0,t]$ with $\mathsf{P}(\gamma >t)>0$ (and over $[0,t_{G})$ if
$\mathsf{P}(\gamma =t_{G}<\infty )>0$). Thus $M$ is regular. It may be that,
if $t_{G}<\infty $ and $\mathsf{P}(\gamma =t_{G})=0$, the function $F$ has
not a finite limit as $t\upuparrows t_{G}$, or, more generally, has unbounded
variation over $[0,t_{G})$. Then a moment of irregularity is given by
\begin{equation*}
\label{eq:irregular}
\rho (\omega )=\left \{
\begin{array}{ll}
t_{G}, & \text{if $\gamma \geqslant t_{G}$;}
\\
+\infty , & \text{otherwise.}
\end{array}
\right .
\end{equation*}
It takes a finite value only on the set $\{\gamma \geqslant t_{G}\}$ of
zero measure. In all other cases we may put $\rho \equiv +\infty $. See
Remark~\ref{re:Fout} in Section~\ref{sec:2} for more details.

%s2 #&#

\section{Main results}
\label{sec:2}

Let $\gamma $ be a random variable with values in
$\overline{\mathbb{R}}_{+}$ on a probability space
$(\Omega ,{\mathscr{F}},\mathsf{P})$.\index{probability space} We tacitly assume that
$\mathsf{P}(\gamma >0)>0$. $G(t)=\mathsf{P}(\gamma \leqslant t)$,
$t\in \mathbb{R}_{+}$, stands for the distribution function of
$\gamma $ and ${\overline{G}}(t)=1-G(t)$. Put also
$t_{G}=\sup \,\{t\in \mathbb{R}_{+}\colon G(t)<1\}$ and
${\mathcal{T}}=\{t\in \mathbb{R}_{+}\colon {\mathsf{P}}(\gamma
\geqslant t)>0\}$. Note that
$\mathsf{P}(\gamma \notin {\mathcal{T}})=0$. We will often distinguish between
the following two cases:
\begin{defenum}
\item[Case~A] $\mathsf{P}(\gamma =t_{G}<\infty )=0$.
\item[Case~B] $\mathsf{P}(\gamma =t_{G}<\infty )>0$.
\end{defenum}
It is clear that ${\mathcal{T}}=[0,t_{G})$ in Case A and
${\mathcal{T}}=[0,t_{G}]$ in Case B.

We define ${\mathscr{F}}_{t}$, $t\in \mathbb{R}_{+}$, as the collection
of subsets $A$ of $\Omega $ such that $A\in {\mathscr{F}}$ and
$A\cap \{t<\gamma \}$ is either $\varnothing $ or coincides with
$\{t<\gamma \}$.

It is shown in Proposition \ref{prop:1} that ${\mathscr{F}}_{t}$ is a
$\sigma $-field for every $t\in \mathbb{R}_{+}$ and the family
$\mathbb{F}=({\mathscr{F}}_{t})_{t\in \mathbb{R}_{+}}$ is a filtration.\index{filtration}
We call this filtration\index{filtration} a \emph{single jump filtration}.\index{single jump ! filtration} It is determined
by generating elements $\gamma $ and ${\mathscr{F}}$. In this paper we
consider only single jump filtrations\index{single jump ! filtration} and, if necessary to indicate generating
elements, we use the notation $\mathbb{F}(\gamma ,{\mathscr{F}})$ for
the single jump filtration\index{single jump ! filtration} generated by $\gamma $ and
${\mathscr{F}}$.

In this section a single jump filtration
$\mathbb{F}=\mathbb{F}(\gamma ,{\mathscr{F}})$\index{single jump ! filtration} is fixed. All notions
depending on filtration\index{filtration} (stopping times,\index{stopping time} martingales,\index{martingale} local martingales,\index{local martingale}
etc.) refer to %are %assumed with respect to %based on
this filtration $\mathbb{F}$,\index{filtration} unless
otherwise specified.

%p1 #&#
\begin{prop}%
\label{prop:1}
\textup{(i)} ${\mathscr{F}}_{t}$ is a $\sigma $-field and a random variable
$\xi $ is ${\mathscr{F}}_{t}$-measurable,
$t\in \mathbb{R}_{+}$, if and only if $\xi $ is constant on
$\{t<\gamma \}$. $\xi $ is ${\mathscr{F}}_{\infty }$-measurable if and only
if $\xi $ is constant on $\{\gamma =\infty \}$.

\textup{(ii)} The family $({\mathscr{F}}_{t})_{t\in \mathbb{R}_{+}}$ is
increasing and right-continuous, i.e.\ $\mathbb{F}=({
\mathscr{F}}_{t})_{t\in \mathbb{R}_{+}}$ is a filtration.\index{filtration}

\textup{(iii)} $\gamma $ is a stopping time\index{stopping time} and
${\mathscr{F}}_{\gamma }={\mathscr{F}}_{\infty }$.

\textup{(iv)} A random variable $T$ with values in
$\overline{\mathbb{R}}_{+}$ is a stopping time\index{stopping time} if and only if it satisfies
the following property: if the set $\{T<\gamma \}$ is not empty,
then there is a number $r$ such that
%
%e4 #&#
\begin{equation}
\label{eq:NS-ST}
\{T<\gamma \}=\{T=r<\gamma \}=\{r<\gamma \}.
\end{equation}
\end{prop}

%p2 #&#
\begin{prop}%
\label{prop:fv}
\textup{(i)} If $X=(X_{t})_{t\in \mathbb{R}_{+}}$ is an adapted process,
then there is a deterministic function $F(t)$, $0\leqslant t<t_{G}$, such
that $X_{t}=F(t)$ on $\{t<\gamma \wedge t_{G}\}$. If
$Y=(Y_{t})_{t\in \mathbb{R}_{+}}$ is an adapted process and
$\mathsf{P}(X_{t}=Y_{t})=1$ for every $t\in \mathbb{R}_{+}$, then
$X_{t}=Y_{t}$ identically on $\{t<\gamma \wedge t_{G}\}$.

\textup{(ii)} If $Y=(Y_{t})_{t\in \mathbb{R}_{+}}$ is a predictable process,
then there is a measurable deterministic function $C(t)$,
$t\in {\mathcal{T}}$, such that $Y_{t}=C(t)$ on
$\{t\leqslant \gamma \}$, $t\in {\mathcal{T}}$.

\textup{(iii)} If $X=(X_{t})_{t\in \mathbb{R}_{+}}$ is a process with finite
variation,\index{finite variation} then $F(t)$ in \textup{(i)} has a finite variation\index{finite variation} over
$[0,t]$ for every $t<t_{G}$ in Case A and over $[0,t_{G})$ in Case~B.

\textup{(iv)} Every semimartingale is a process with finite variation.\index{finite variation}

\textup{(v)} If $M=(M_{t})_{t\in \mathbb{R}_{+}}$ is a $\sigma $-martingale
then there are a deterministic function $F(t)$,
$t\in \mathbb{R}_{+}$, and a finite random variable $L$ such that,
up to $\mathsf{P}$-indistinguishability,
%
%e5 #&#
\begin{equation}
\label{eq:mr}
M_{t}=F(t){\Eins }_{\{t<\gamma \}}+L{\Eins }_{\{t\geqslant
\gamma \}}.
\end{equation}
\end{prop}

Statement (iv) is not surprising. If the $\sigma $-field
${\mathscr{F}}$ is countably generated, then our filtration\index{filtration} is a special
case of a filtration generated by a marked point process, and it is known,
see \cite{Jacod1979}, that then all martingales are of finite variation.\index{finite variation}
In general, a single jump filtration\index{single jump ! filtration} is a special case of a jumping filtration,
see \cite{JacodSkorokhod1994}, where again all martingales are of finite
variation.\index{finite variation}

%r1 #&#
\begin{remark}%
\label{re:Ffv}
If $M$ is a $\sigma $-martingale, then it is a process with finite variation\index{finite variation}
due to (iv) and, hence, the function $F(t)$ in \eqref{eq:mr} has a finite
variation\index{finite variation} over $[0,t]$ for every $t<t_{G}$ in Case A and over
$[0,t_{G})$ in Case B according to (iii).
\end{remark}

%r2 #&#
\begin{remark}%
\label{re:Fout}
According to (i), the function $F(t)$ in \eqref{eq:mr} is uniquely determined
for $t<t_{G}$. Since $\mathsf{P}(\gamma > t_{G})=0$, the stochastic interval
$\llbracket t_{G},\gamma \rlb $ is an evanescent set.\index{evanescent set} Hence, $F(t)$ can
be defined arbitrarily for $t\geqslant t_{G}$. For example, we can put
it equal to $0$ for $t\geqslant t_{G}$. Then $F(t)$ has a finite variation\index{finite variation}
on compact intervals if $t_{G}=+\infty $ or in Case B. In Case A, if
$t_{G}$ is finite, $F(t)$ may have infinite variation\index{infinite variation} over
$[0,t_{G})$ (and even not have a finite limit as
$t\upuparrows t_{G}$), see Theorem~\ref{th:solutions} and Example~\ref{ex:usualcond}
below. All other points are regular for $F(t)$. Now put
$\rho (\omega )=t_{G}<+\infty $ if we are in Case A,
$t_{G}<+\infty $, $\lim _{t\upuparrows t_{G}}\Var (F)_{t}=\infty $, and
$\gamma (\omega )\geqslant t_{G}$, and let $\rho (\omega )=+\infty $ in
all other cases. It follows that $\rho $ is a moment of irregularity for
the process in the right-hand side of \eqref{eq:mr}.
\end{remark}

In what follows, when we write that the process $M$ has the representation \eqref{eq:mr}, this means that $M$ and the right-hand side of \eqref{eq:mr} are indistinguishable. Moreover, we tacitly assume that
$F(t)$ is right-continuous for $t\geqslant t_{G}$ to ensure that the right-hand
side of \eqref{eq:mr} is right-continuous.

Propositions \ref{prop:1} and \ref{prop:fv} explain why we call
$\mathbb{F}$ a single jump filtration:\index{single jump ! filtration} all randomness appears at time
$\gamma $. It is not so natural to describe local martingales\index{local martingale} with respect
to $\mathbb{F}$ as single jump processes. As we will see, the function
$F$ in \eqref{eq:mr} need not be continuous, so local martingales\index{local martingale} may have
several jumps.

Our main goal is to provide a complete description of all local martingales.\index{local martingale}
According to Proposition \ref{prop:fv} (v), a necessary condition is that
it is represented in form \eqref{eq:mr}. Thus, it is enough to study only
processes of this form.

%t1 #&#
\begin{thm}%
\label{th:2}
Let $F(t)$, $0\leqslant t<t_{G}$, be a deterministic c\`{a}dl\`{a}g function,\index{c\`{a}dl\`{a}g ! function}
$L$ be a random variable, and a process
$M=(M_{t})_{t\in \mathbb{R}_{+}}$ be given by
%
%e6 #&#
\begin{equation}
\label{eq:mr2}
M_{t}=F(t){\Eins }_{\{t<\gamma \}}+L{\Eins }_{\{t\geqslant
\gamma \}}.
\end{equation}
The following statements are equivalent:
\begin{enumerate}
\item[\textup{(i)}] $M=(M_{t})_{t\in \mathbb{R}_{+}}$ is a local martingale.\index{local martingale}
\item[\textup{(ii)}] $(M_{t})_{t\in {\mathcal{T}}}$ is a martingale.\index{martingale}
\item[\textup{(iii)}]
%
%e7 #&#
\begin{equation}
\label{eq:integrab}
{\mathsf{E}}\bigl (|M_{t}| \bigr )<\infty , \quad t\in {\mathcal{T}},
\end{equation}
and
%
%e8 #&#
\begin{equation}
\label{eq:main+}
{\mathsf{E}}(M_{t}) = \mathsf{E}(M_{0}), \quad t\in {\mathcal{T}}.
\end{equation}
\end{enumerate}
\end{thm}

In the case where ${\mathscr{F}}=\sigma \{\gamma \}$, equivalence (i)
and (ii) is proved in~\cite{ChouMeyer:1975}.

Concerning the last statement of the proposition, let us emphasize that
if $t_{G}<\infty $ and $\mathsf{P}(\gamma =t_{G})=0$, a local martingale
$M=(M_{t})_{t\in \mathbb{R}_{+}}$\index{local martingale} may not be a martingale\index{martingale} on
$[0,t_{G}]$; obviously, if it is a martingale,\index{martingale} then it is uniformly integrable,
and necessary and sufficient conditions for this are given in Theorem~\ref{th:class}.

If \eqref{eq:mr2} and \eqref{eq:integrab} hold, then
%
%e9 #&#
\begin{equation}
\label{eq:integrabL}
{\mathsf{E}}\bigl (|L| {\Eins }_{\{\gamma \leqslant t\}}\bigr )<\infty ,
\quad t\in {\mathcal{T}},
\end{equation}
and one can define the conditional expectation $H(t)$ of $L$ given that
$\gamma =t$ for $t\in {\mathcal{T}}$:
%
%e10 #&#
\begin{equation}
\label{eq:Hdef1}
H(t)=\mathsf{E}[L|\gamma =t].
\end{equation}
More precisely, $H(t)$ is a Borel function on ${\mathcal{T}}$ with finite
values such that for any $t\in {\mathcal{T}}$
\begin{equation*}
\label{eq:Hdef2}
{\mathsf{E}}\bigl (L{\Eins }_{\{\gamma \leqslant t\}}\bigr ) = \int _{[0,t]}
H(s)\,dG(s).
\end{equation*}
Note that the function $H$ is $dG$-a.s.\ unique and is $dG$-integrable
over any closed interval in ${\mathcal{T}}$. It is convenient to introduce
a notation for such functions.

Let $L^{1}_{\mathrm{loc}}(dG)$ be the set of all Borel functions $z$ on
${\mathcal{T}}$ such that
\begin{equation*}
\int _{[0,t]} |z(s)|\,dG(s)<\infty \quad
\text{for all $t\in {\mathcal{T}}$}.
\end{equation*}
Given a function $Z\colon [0,t_{G})\to \mathbb{R}$, let us write
$Z\overset{\mathrm{loc}}{\ll }G$ if there is $z\in L^{1}_{\mathrm{loc}}(dG)$ such
that $Z(t)=Z(0)+\int _{(0,t]} z(s)\,dG(s)$ for all $t<t_{G}$; in this case
we put $\frac{dZ}{dG}(t):=z(t)$ for $0<t<t_{G}$. Let us emphasize that
in Case B this definition implies that $z$ is $dG$-integrable over
$[0,t_{G}]$ and, hence, the function $Z$ has a finite variation\index{finite variation} over
$[0,t_{G})$ and there is a finite limit
$\lim _{t\upuparrows t_{G}} Z(t) = Z(0) + \int _{(0,t_{G})}z(s)\,dG(s)$.
Note also that in this definition the value $z(0)$ can be chosen arbitrarily
even if $G(0)>0$; the same refers to the value $z(t_{G})$ in Case B. Correspondingly,
$dZ/dG$ is defined only for $0<t<t_{G}$.

Let $G$ be a distribution function of a law on $[0,+\infty ]$. We will
say that a pair $(F,H)$ satisfies Condition M if
%
%e11 #&#
\begin{gather}
\label{eq:MF}
F\colon [0,t_{G})\to \mathbb{R}, \quad F\overset{\mathrm{loc}}{\ll }G,
\\
\label{eq:MH}
H\colon {\mathcal{T}}\to \mathbb{R}, \quad H\in L^{1}_{\mathrm{loc}}(dG),
\\
\label{eq:main}
F(t){\overline{G}}(t) + \int \limits _{(0,t]}H(s)\,dG(s) = F(0){
\overline{G}}(0), \quad t<t_{G},
\end{gather}
and, additionally in Case B,
%
%e14 #&#
\begin{equation}
\label{eq:mainB}
\lim _{t\upuparrows t_{G}} F(t) = H(t_{G}).
\end{equation}

%p3 #&#
\begin{prop}%
\label{prop:solutions}
\textup{(a)} Let $H$ be any function satisfying \eqref{eq:MH}. Define
%
%e15 #&#
\begin{equation}
\label{eq:F}
F(t)= {\overline{G}}(t)^{-1} \Bigl [F(0){\overline{G}}(0) - \int
\limits _{(0,t]} H(s)\,dG(s)\Bigr ],\quad 0<t<t_{G},
\end{equation}
where $F(0)$ is an arbitrary real number in Case A and
%
%e16 #&#
\begin{equation}
\label{eq:F0}
F(0) = {\overline{G}}(0)^{-1}\int \limits _{(0,t_{G}]}H(s)\,dG(s)
\end{equation}
in Case B. Then the pair $(F,H)$ satisfies Condition M. Conversely, if
$F$ is such that the pair $(F,H)$ satisfies Condition M, then $F$ satisfies \eqref{eq:F} and, in Case B, \eqref{eq:F0} holds.

\textup{(b)} Let $F$ be any function satisfying \eqref{eq:MF}. Define
$H(0)$ arbitrarily,
%
%e17 #&#
\begin{equation}
\label{eq:H}
H(t)=F(t)-{\overline{G}}(t-)\frac{dF}{dG}(t),\quad 0< t<t_{G},
\end{equation}
$H(t_{G})$ arbitrarily in Case A and
%
%e18 #&#
\begin{equation}
\label{eq:Ht}
H(t_{G})=\lim _{t\upuparrows t_{G}} F(t)
\end{equation}
in Case B. Then the pair $(F,H)$ satisfies Condition M. Conversely, if
$H$ is such that the pair $(F,H)$ satisfies Condition M, then $H$ satisfies \eqref{eq:H} and, in Case B, \eqref{eq:Ht} holds.
\end{prop}

%t2 #&#
\begin{thm}%
\label{th:solutions}
In order that a right-continuous process
$M=(M_{t})_{t\in \mathbb{R}_{+}}$ be a local martingale\index{local martingale} it is necessary
and sufficient that there be a pair $(F,H)$ satisfying Condition M and
a random variable $L'$ satisfying
%
%e19 #&#
\begin{equation}
\label{eq:integrab2}
{\mathsf{E}}\bigl (|L'| {\Eins }_{\{\gamma \leqslant t\}}\bigr )<\infty ,
\quad t\in {\mathcal{T}}, \qquad \text{and}\qquad \mathsf{E}[L'|\gamma ]=\xch{0,}{0.}
\end{equation}
such that, up to $\mathsf{P}$-indistinguishability,
%
%e20 #&#
\begin{equation}
\label{eq:mr3}
M_{t}=F(t){\Eins }_{\{t<\gamma \}}+\bigl (H(\gamma )+L'\bigr ){
\Eins }_{\{t\geqslant \gamma \}},
\end{equation}
\end{thm}

The statement that the process $M$ given by \eqref{eq:mr3} with
$L'=0$ is a local martingale\index{local martingale} if $F\overset{\mathrm{loc}}{\ll }G$ and $H$ is constructed
as in part (b) of Proposition~\ref{prop:solutions}, is essentially due
to Herdegen and Herrmann \cite{HerdegenHerrmann:16}, though they formulate \eqref{eq:H} in an equivalent form:
%
%e21 #&#
\begin{equation}
\label{eq:H++}
H(t)=F(t-)-{\overline{G}}(t)\frac{dF}{dG}(t),\quad 0< t<t_{G}.
\end{equation}
They also prove that, in Case B, if $F$ has infinite variation on
$[0,t_{G})$ (and hence does not satisfy $F\overset{\mathrm{loc}}{\ll }G$), then
$M$ given by \eqref{eq:mr2} is not a semimartingale,\index{semimartingale} see
\cite[Lemma B.6]{HerdegenHerrmann:16}. (Note that this follows also from
our Proposition~\ref{prop:fv} (iv).) We add that, also in Case B, if
$H$ is $dG$-integrable over $(0,t_{G})$, $F$ satisfies \eqref{eq:F}, but
$F(0)$ is greater or less than the right-hand side of \eqref{eq:F0}, then
$M$ given by \eqref{eq:mr3} with $L'$ satisfying \eqref{eq:integrab2},
is a supermartingale\index{supermartingale} or a submartingale,\index{submartingale} respectively, cf.\ Theorem \ref{th:class}.

The fact that $H(0)$ can be chosen arbitrarily in Proposition~\ref{prop:solutions}
(b) says only that $L$ can be an arbitrary integrable random variable on
the set $\{\gamma =0\}$, which is evident ab initio. On the contrary, the
fact that $F(0)$ can be chosen arbitrarily in (a) in Case~A is an interesting
feature of this model. It says that, given the terminal value
$M_{\infty }$\index{terminal value} of $M$ (on $\{\gamma <\infty \}$), one can freely choose the
initial value $M_{0}$ of $M$ (on $\{\gamma >0\}$) to keep the property
of being a local martingale\index{local martingale} for $M$.

%c1 #&#
\begin{cor}%
\label{co:decomp}
Every local martingale\index{local martingale} $M=(M_{t})_{t\in \mathbb{R}_{+}}$ has a unique
decomposition into the sum $M=M'+M''$ of two local martingales\index{local martingale} $M'$ and
$M''$, where $M'$ is adapted with respect to the smallest filtration\index{filtration} making
$\gamma $ a stopping time,\index{stopping time} and $M''$ which vanishes on
$\{t<\gamma \}$ and satisfies $\mathsf{E}M''_{0} = 0$.
\end{cor}

%r3 #&#
\begin{remark}%
\label{re:H0=0}
If $\mathsf{P}(\gamma =0)=0$, then it follows from the first property for
$M''$ that $M''_{0}=0$ a.s. and thus the second property holds automatically.
\end{remark}

%r4 #&#
\begin{remark}%
\label{re:decomp}
The smallest filtration\index{filtration} making $\gamma $ a stopping time\index{stopping time} is a single jump
filtration $\mathbb{F}(\gamma ,\sigma \{\gamma \})$\index{single jump ! filtration} generated by
$\gamma $ and the smallest $\sigma $-field $\sigma \{\gamma \}$ with respect
to which $\gamma $ is measurable. Let $M$ be a $\mathbb{F}$-local martingale
adapted to $\mathbb{F}(\gamma ,\sigma \{\gamma \})$. It follows from Theorem~\ref{th:2}
that $M$ is a $\mathbb{F}(\gamma ,\sigma \{\gamma \})$-local martingale.
\end{remark}

As the next example shows, the product $M'M''$ of local martingales\index{local martingale} from
the above decomposition may not be a local martingale\index{local martingale} because the first
condition in \eqref{eq:integrab2} may fail. It will follow from Theorem \ref{th:sigmamart} below that this product is always a $\sigma $-martingale.

%e1 #&#
\begin{example}%
\label{ex:Emery}
Let $\gamma $ have an exponential distribution, e.g.,
${\overline{G}}(t)=e^{-t}$, $F$ is given by \eqref{eq:F} with
$H(t)=t^{-1/2}$ and an arbitrary $F(0)$,
$M'_{t}=F(t){\Eins }_{\{t<\gamma \}}+H(\gamma ){\Eins }_{\{t
\geqslant \gamma \}}$,
$M''_{t}=Y\gamma ^{-1/2}{\Eins }_{\{t\geqslant \gamma \}}$, where
$Y$ takes values $\pm 1$ with probabilities $1/2$ and is independent of
$\gamma $. It follows that $M'$ and $M''$ are local martingales\index{local martingale} but their
product
$M'_{t}M''_{t}=Y\gamma ^{-1}{\Eins }_{\{t\geqslant \gamma \}}$ does
not satisfy the integrability condition \eqref{eq:integrab} and, hence,
is not a local martingale.\index{local martingale} This process is a classical example (due to
\'{E}mery) of a $\sigma $-martingale which is not a local martingale,\index{local martingale} see,
e.g., \cite[Example 2.3, p.~86]{Gushchin:15}.
\end{example}

The previous example shows that our model admits $\sigma $-martingales
which are not local martingales.\index{local martingale} In the next theorem we describe all
$\sigma $-martingales in our model. In particular, it implies that
if ${\mathscr{F}}=\sigma \{\gamma \}$, then all $\sigma $-martingales
that are integrable at $0$ are local martingales.\index{local martingale}

%t3 #&#
\begin{thm}%
\label{th:sigmamart}
In order that a right-continuous process
$M=(M_{t})_{t\in \mathbb{R}_{+}}$ be a $\sigma $-martin\-gale it is necessary
and sufficient that it have a representation \eqref{eq:mr3}, where
a pair $(F,H)$ satisfies Condition M and a random variable $L'$ satisfies
%
%e22 #&#
\begin{equation}
\label{eq:integrabs}
{\mathsf{E}}\bigl [|L'|{\Eins }_{\{\gamma >0\}} \big |\gamma \bigr ]<
\infty \qquad \text{and}\qquad \mathsf{E}[L'|\gamma ]=0.
\end{equation}
\end{thm}

The next theorem complements the classification of the limit behaviour
of local martingales\index{local martingale} that was considered in Herdegen and Herrmann
\cite{HerdegenHerrmann:16} in the case where
${\mathscr{F}}=\sigma \{\gamma \}$. Let us say that a local martingale\index{local martingale}
$M=(M_{t})_{t\in \mathbb{R}_{+}}$ has
\begin{defenum}
\item[type 1] if the limit $M_{\infty }=\lim _{t\to \infty }M_{t}$ does not
exist with positive probability or exists with probability one but is not
integrable: $\mathsf{E}|M_{\infty }|= \infty $;
\item[type 2a] if $M$ is a closed supermartingale (in particular,
$\mathsf{E}|M_{\infty }|< \infty $) and
$\mathsf{E}(M_{\infty })< \mathsf{E}(M_{0})$;
\item[type 2b] if $M$ is a closed submartingale (in particular,
$\mathsf{E}|M_{\infty }|< \infty $) and
$\mathsf{E}(M_{\infty })> \mathsf{E}(M_{0})$;
\item[type 3] if $M$ is a uniformly integrable martingale (in particular,
$\mathsf{E}|M_{\infty }|< \infty $ and
$\mathsf{E}(M_{\infty })= \mathsf{E}(M_{0})$) and
$\mathsf{E}(\sup _{t} |M_{t}|)=\infty $;
\item[type 4] if $M$ has an integrable variation:
$\mathsf{E}\bigl (\Var (M)_{\infty }\bigr ) <\infty $.
\end{defenum}

%t4 #&#
\begin{thm}%
\label{th:class}
Let $M=(M_{t})_{t\in \mathbb{R}_{+}}$ be a local martingale\index{local martingale} with the representation
%
%e23 #&#
\begin{equation}
\label{eq:mrU}
M_{t}=F(t){\Eins }_{\{t<\gamma \}}+L{\Eins }_{\{t\geqslant
\gamma \}},\quad t\in \mathbb{R}_{+},
\end{equation}
where $L=H(\gamma )+L'$, a pair $(F,H)$ satisfies Condition M and a random
variable $L'$ satisfies \eqref{eq:integrab2}. Then in Case B the local
martingale $M$\index{local martingale} has type\/ {\textup{4}}. In Case A all types are possible. Namely,
\begin{enumerate}
\item[\textup{(i)}] $M$ has type\/ {\textup{1}} if and only if
$\mathsf{E}\bigl (|L'|{\Eins }_{\{\gamma <\infty \}}\bigr )=\infty $ or
$\int _{[0,t_{G})} |H(s)|\,dG(s)=\infty $.
\item[\textup{(ii)}] If $\mathsf{P}(\gamma =\infty )>0$,
$\mathsf{E}\bigl (|L'|{\Eins }_{\{\gamma <\infty \}}\bigr )<\infty $, and
$\int _{\mathbb{R}_{+}} |H(s)|\,dG(s)<\infty $ then $M$ has type\/ {\textup{4}}.
\item[\textup{(iii)}] If $\mathsf{P}(\gamma =\infty )=0$,
$\mathsf{E}|L'|<\infty $, and $\int _{[0,t_{G})} |H(s)|\,dG(s)<\infty $ then
\begin{enumerate}
\item[\textup{(iii.i)}] $M$ has type\/ {\textup{2a (}}resp., \textup{2b)} if and only
if $\lim _{t\upuparrows t_{G}} F(t){\overline{G}}(t) > 0$ \textup{(}resp.,
$\lim _{t\upuparrows t_{G}} F(t){\overline{G}}(t) < 0)$;
\item[\textup{(iii.ii)}] $M$ has type\/ {\textup{3}} if and only if
\begin{equation*}
\lim _{t\upuparrows t_{G}} F(t){\overline{G}}(t) = 0 \qquad \text{and}
\qquad \int _{[0,t_{G})} {\overline{G}}(s)\Bigl |\frac{dF}{dG}(s)\Bigr |
\,dG(s)=\infty ;
\end{equation*}
\item[\textup{(iii.iii)}] $M$ has type\/ {\textup{4}} if and only if
%
%e24 #&#
\begin{equation}
\label{eq:H1}
\int \limits _{[0,t_{G})} {\overline{G}}(s)\Bigl |\frac{dF}{dG}(s)
\Bigr |\,dG(s)<\infty .
\end{equation}
\end{enumerate}
\end{enumerate}
\end{thm}

%r5 #&#
\begin{remark}
It follows from \eqref{eq:main} that the limit
$\lim _{t\to t_{G}} F(t){\overline{G}}(t)$ in (iii.i) and (iii.ii) exists.
Also, $\int _{[0,t_{G})} |H(s)|\,dG(s)$ in (i)--(iii) is finite if only
if $F(t){\overline{G}}(t)$ has a finite variation\index{finite variation} over $[0,t_{G})$.
\end{remark}

%r6 #&#
\begin{remark}
It follows from Theorem \ref{th:class} that, in our model, every martingale
$M$ with $\mathsf{E}(\sup _{t} |M_{t}|)<\infty $ has an integrable total variation.
Of course, on general spaces, there exist martingales $M$\index{martingale} having finite
variation on compacts and such that\break
$\mathsf{E}(\sup _{t} |M_{t}|)<\infty $ and their total variation is not integrable,
see, e.g., \cite[Example 2.7, p.~103]{Gushchin:15}.
\end{remark}

%e2 #&#
\begin{example}%
\label{ex:D-G}
Assume that $H\colon (0,1)\to \mathbb{R}$ is a
\emph{monotone nondecreasing\/} function and, for definiteness, that it
is right-continuous. Then it is the upper quantile function of
$H(\gamma )$, where $\gamma $ is uniformly distributed on $(0,1)$. Assume
also that $H$ is integrable on $(0,1)$ and
$\int _{0}^{1} H(s)\,ds=0$, that is to say, that $H(\gamma )$ has zero mean.
Put
\begin{equation*}
F(t)= -(1-t)^{-1} \int \limits _{0}^{t} H(s)\,ds = (1-t)^{-1} \int
\limits _{t}^{1} H(s)\,ds.
\end{equation*}
We see that $F$ satisfying \eqref{eq:main} with $F(0)=0$ is the Hardy--Littlewood
maximal function corresponding to $H$. If we define $M$ by \eqref{eq:mrU} with $L=H(\gamma )$, then, by
Theorem~\ref{th:class},
$M$ is a uniformly integrable martingale with $M_{\infty }=H(\gamma )$ and
$\sup _{t} M_{t} = F(\gamma )$. This example is essentially the example
of Dubins and Gilat \cite{DubinsGilat} of a uniformly integrable martingale
with a given distribution of its terminal value, having maximal (with respect
to the stochastic partial order) maximum (in time).
\end{example}

%e3 #&#
\begin{example}[\xch{{\cite[Example 3.14]{HerdegenHerrmann:16}}}{\cite{HerdegenHerrmann:16}, Example 3.14}]%
\label{ex:usualcond}
Let $\Omega =(0,1]$ be equipped with the Borel $\sigma $-field
${\mathscr{F}}$, and let $\mathsf{P}$ be the Lebesgue measure,
$\gamma (\omega )=\omega $. Put $H(t)\equiv 0$. Then
$F(t)=(1-t)^{-1}$ satisfies \eqref{eq:main} with $F(0)=1$. By Theorem~\ref{th:class},
$M$ defined by \eqref{eq:mrU} is a supermartingale\index{supermartingale} and local martingale\index{local martingale}
but not a martingale.\index{martingale} This seems to be the simplest example of a local
martingale\index{local martingale} with continuous time, which is not a martingale.\index{martingale} Note that,
for $\omega =1$, the trajectory
$M_{t}(\omega )=(1-t)^{-1}{\Eins }_{\{t<1\}}$ has not a finite left-hand
limit at $1$. Moreover, if $N$ is a modification of $M$, for $t<1$, the
values of $M_{t}(\omega )$ and $N_{t}(\omega )$ must coincide on the atom
$\{t<\gamma \}=(t,1]$ of ${\mathscr{F}}_{t}$, having the positive measure.
Hence, $N_{t}(\omega )=M_{t}(\omega )$ for $\omega =1$ for all $t<1$. This
is an example of a right-continuous supermartingale which has not a modification
with\/ \emph{all} \xch{paths c\`{a}dl\`{a}g}{c\`{a}dl\`{a}g paths}. Of course, the usual assumptions are
not satisfied in this example.
\end{example}

%s3 #&#

\section{Proofs}
\label{sec:4}

\begin{proof}[Proof of Proposition~\ref{prop:1}]
(i) and (iii) are evident from the definition of
${\mathscr{F}}_{t}$, and (ii) follows easily from (i).

Let us prove (iv). To prove that $T$ is a stopping time,\index{stopping time} we must check
that $\{T\leqslant t<\gamma \}$ is either $\varnothing $ or
$\{t<\gamma \}$ for all $t\in \mathbb{R}_{+}$. This is trivial if
$\{T<\gamma \}=\varnothing $. If there is a number $r$ such that \eqref{eq:NS-ST} holds, then $\{T\leqslant t<\gamma \}$ is either
$\varnothing $ if $r>t$ or $\{t<\gamma \}$ if $r\leqslant t$.

Conversely, let $T$ be a stopping time.\index{stopping time} If $T\geqslant \gamma $ for all
$\omega $, then there is nothing to prove. Assume that the set
$\{T<\gamma \}\neq \varnothing $. Then there are real numbers $q$ such
that $\{T\leqslant q<\gamma \}\neq \varnothing $. For such $q$, by the
definition of ${\mathscr{F}}_{q}$,
$\{T\leqslant q<\gamma \}=\{q<\gamma \}$, or, equivalently,
$\{T\leqslant q\}\supseteq \{q<\gamma \}$. Let $r$ be the greatest lower
bound of such $q$. The sets $\{q<\gamma \}\uparrow \{r<\gamma \}$ and
$\{T\leqslant q\}\downarrow \{T\leqslant r\}$ as $q\downarrow r$. Thus,
\begin{equation*}
\{T<\gamma \}=\bigcup _{q\colon \{T\leqslant q<\gamma \}\neq
\varnothing }\{q<\gamma \}=\{r<\gamma \}\subseteq \{T\leqslant r\}.
\end{equation*}
Since $\{T\leqslant t<\gamma \}=\varnothing $ for any $t<r$, we have \eqref{eq:NS-ST}.
\end{proof}

\begin{proof}[Proof of Proposition~\ref{prop:fv}]

The first statement in (i) follows from Proposition~\ref{prop:1} (i). Since
$\mathsf{P}(t<\gamma \wedge t_{G})>0$ for every $t<t_{G}$, we obtain that
$X_{t}$ and $Y_{t}$ take the same constant value on
$\{t<\gamma \wedge t_{G}\}$.

Since a random variable $Y_{t}$ is ${\mathscr{F}}_{t-}$-measurable for
a predictable process $Y$, $Y_{t}$ is constant on
$\{t\leqslant \gamma \}$. Denote by $C(t)$, $t\in {\mathcal{T}}$, the
value of $Y_{t}$ on $\{t\leqslant \gamma \}$. Since
$\mathsf{P}(\gamma \geqslant t)>0$ for $t\in {\mathcal{T}}$, there is an
$\omega $ such that $C(s)\equiv Y_{s}(\omega )$, $s\leqslant t$, and the
measurability of $C$ follows.

Let us prove (iii) in Case B. Then we obtain that $X_{t}=F(t)$ for all
$t<t_{G}$ on the set $\{\gamma =t_{G}\}$, which has a positive probability.
However, almost all paths of $X_{t}$ have a finite variation\index{finite variation} over
$[0,t_{G})$, and the claim follows. The proof in Case A is similar.

Now let us prove \eqref{eq:mr} in the case where
$M=(M_{t})_{t\in \mathbb{R}_{+}}$ is a uniformly integrable (a.s. c\`{a}dl\`{a}g\index{c\`{a}dl\`{a}g})
martingale.\index{martingale} We can find a random variable $M_{\infty }$ that is
${\mathscr{F}}_{\infty }$-measurable and such that
$\lim _{n\to \infty }M_{n} = M_{\infty }$ $\mathsf{P}$-a.s. Since
$\{t<\gamma \}$ is an atom of ${\mathscr{F}}_{t}$ and has a positive probability
for $t<t_{G}$, we obtain from the martingale property that
$M_{t}(\omega )= F(t)$ for all $\omega \in \{t<\gamma \}$, where
\begin{equation*}
\label{eq:F(t)}
F(t)=
\frac{\mathsf{E}\bigl (M_{\infty }{\Eins }_{\{t<\gamma \}}\bigr )}{{\overline{G}}(t)},
\quad t<t_{G}.
\end{equation*}
It is clear that the nominator and the denominator are right-continuous
functions of bounded variation on $[0,t_{G}]$, hence $F(t)$,
$0\leqslant t<t_{G}$ is a c\`{a}dl\`{a}g function\index{c\`{a}dl\`{a}g ! function} on ${\mathcal{T}}$ and has
a finite variation\index{finite variation} on $[0,t_{G})$ in Case B and on every $[0,t]$,
$0\leqslant t<t_{G}$, in Case A.

Now set $L=M_{\infty }{\Eins }_{\{\gamma <\infty \}}$. Then
$L{\Eins }_{\{\gamma \leqslant t\}}=M_{\infty }{\Eins }_{\{
\gamma \leqslant t\}}$ is ${\mathscr{F}}_{t}$-measurable, and hence
$\mathsf{P}$-a.s.
\begin{equation*}
M_{t}{\Eins }_{\{\gamma \leqslant t\}}=\mathsf{E}(M_{\infty }{
\Eins }_{\{\gamma \leqslant t\}}|{\mathscr{F}}_{t})=L{\Eins }_{
\{\gamma \leqslant t\}}.
\end{equation*}
Thus we have obtained, that, for a given $t\in \mathbb{R}_{+}$,
$M_{t}$ is equal $\mathsf{P}$-a.s. to the right-hand side of \eqref{eq:mr} with $L$ and $F(t)$ as above. Since both the left-hand side
and the right-hand side of \eqref{eq:mr} are almost surely right-continuous,
they are indistinguishable. Moreover, if we change $F(t)$ for
$t\geqslant t_{G}$, the right-hand side of \eqref{eq:mr} will change on
an evanescent set.\index{evanescent set} Thus we can put, say, $F(t)=0$ for
$t\geqslant t_{G}$, and then the right-hand side of \eqref{eq:mr} is a
regular right-continuous process with finite variation,\index{finite variation} and indistinguishable
from $M$.

Now let $M$ be a local martingale\index{local martingale} and $\{T_{n}\}$ be a localizing sequence
of stopping times,\index{stopping time} i.e.\ $T_{n}\uparrow \infty $ a.s.\ and
$M^{T_{n}}$ is a uniformly integrable martingale for each $n$. We have
proved that almost all paths of $M^{T_{n}}$ have finite variation.\index{finite variation} It follows
that almost all paths of $M$ have finite variation.\index{finite variation} This proves (iv).

Next, let $M$ be a $\sigma $-martingale, %that is
i.e. $M$ is a semimartingale\index{semimartingale}
and there is an increasing sequence of predictable sets\index{predictable sets}
$\Sigma _{n}$ such that
$\cup _{n} \Sigma _{n}=\Omega \times \mathbb{R}_{+}$ and the integral
process ${\Eins }_{\Sigma _{n}}\cdot M$ is a uniformly integrable martingale
for every $n$. It does not matter if we integrate over $[0,t]$ or
$(0,t]$, so let us agree that the domain of integration does not include
$0$. Since the integrand is bounded and every semimartingale\index{semimartingale} is a process
with finite variation\index{finite variation} in our model, the integral can be considered in the
Lebesgue--Stieltjes sense, as well as other integrals appearing in the
proof. Since ${\Eins }_{\Sigma _{n}}\cdot M$ is stopped at
$\gamma $ for every $n$ with probability one, we have
\begin{equation*}
\int {\Eins }_{\lrb \gamma ,\infty \rlb \cap \Sigma _{n}}(t)\,d
\Var (M)_{t} = \int {\Eins }_{\lrb \gamma ,\infty \rlb }(t)\,d
\Var ({\Eins }_{\Sigma _{n}}\cdot M)_{t} = 0 \quad \mathsf{P}
\text{-a.s.}
\end{equation*}
for every $n$, therefore,
\begin{equation*}
\int {\Eins }_{\lrb \gamma ,\infty \rlb }(t)\,d\Var (M)_{t} = 0
\quad \mathsf{P}\text{-a.s.}
\end{equation*}
Combining with (i), we prove representation \eqref{eq:mr}.
\end{proof}

%r7 #&#
\begin{remark}%
\label{re:UImod}
As it was already explained in the introduction, we can prove directly,
without assuming that paths are a.s. c\`{a}dl\`{a}g,\index{c\`{a}dl\`{a}g} that any uniformly integrable
martingale has a regular modification. The proof is essentially the same
as above where we proved that a.s. c\`{a}dl\`{a}g\index{c\`{a}dl\`{a}g} uniformly integrable martingale
has representation \eqref{eq:mr}.
\end{remark}

\begin{proof}[Proof of Theorem~\ref{th:2}]
First, we prove that statements (ii) and (iii) are equivalent. The implication
(ii)$\Rightarrow $(iii) follows trivially from the definition of a martingale.\index{martingale}
Conversely, let (iii) hold. The process
$(M_{t})_{t\in {\mathcal{T}}}$ is right-continuous, adapted by Proposition \ref{prop:1} (i), and integrable, see \eqref{eq:integrab}. Moreover, due
to \eqref{eq:mr2},
\begin{equation*}
M_{t}-M_{s} = 0 \quad \text{on}\ \{s\geqslant \gamma \},
\end{equation*}
where $0\leqslant s<t\in {\mathcal{T}}$. Hence,
\begin{equation*}
{\mathsf{E}}[M_{t}-M_{s}|{\mathscr{F}}_{s}] = 0 \quad \text{on}\ \{s
\geqslant \gamma \}.
\end{equation*}
But $\mathsf{E}[M_{t}-M_{s}|{\mathscr{F}}_{s}]$ is
${\mathscr{F}}_{s}$-measurable and, thus, equals a constant on
$\{s <\gamma \}$. And this constant must be zero since
$\mathsf{E}(M_{t}-M_{s})=0$ by \eqref{eq:main+}.

The implication (ii)$\Rightarrow $(i) is trivial if $t_{G}=\infty $ or
$t_{G}\in {\mathcal{T}}$. So we assume that $t_{G}<\infty $ and
${\overline{G}}(t)\downarrow 0$ as $t\upuparrows t_{G}$. Let
$t_{1}<\cdots <t_{n}<\cdots <t_{G}$, $t_{n}\to t_{G}$, be an increasing sequence,
then ${\overline{G}}(t_{n})\to 0$. Put
\begin{equation*}
T_{n} = \left \{
\begin{array}{ll}
t_{n}, & \text{if $\gamma >t_{n}$;}
\\
+\infty , & \text{otherwise.}
\end{array}
\right .
\end{equation*}
Then $T_{n}$ is a stopping time\index{stopping time} by Proposition \ref{prop:1} (iv),
$T_{n}\uparrow \infty $ a.s., and
$M_{t\wedge T_{n}} = M_{t\wedge t_{n}}$. Hence, $M^{T_{n}}$ is a martingale\index{martingale}
and $M$ is a local martingale.\index{local martingale}

It remains to prove the implication (i)$\Rightarrow $(ii). Let
$M=(M_{t})_{t\in \mathbb{R}_{+}}$ be a local martingale\index{local martingale} with a localizing
sequence $\{T_{n}\}$, i.e.\ $T_{n}\uparrow \infty $ a.s.\ and
$M^{T_{n}}$ is a uniformly integrable martingale for each $n$. If
$\mathsf{P}(T_{n}\geqslant \gamma )=1$ for some $n$, then $M=M^{T_{n}}$ is
a uniformly integrable martingale, and there is nothing to prove. So assume
that $\mathsf{P}(T_{n}<\gamma )>0$ for all $n$. By Proposition \ref{prop:1} (iv), there is a number $r_{n}$ such that
$\{T_{n}<\gamma \}=\{T_{n}=r_{n}<\gamma \}=\{r_{n}<\gamma \}$. It follows
from $\mathsf{P}(r_{n}<\gamma )>0$ that $r_{n}<t_{G}$. In Case B we get
$\mathsf{P}(T_{n}<\gamma ) = \mathsf{P}(r_{n}<\gamma )\geqslant {\mathsf{P}}(
\gamma =t_{G})>0$ for every $n$, a contradiction with
$T_{n}\to \infty $ a.s. In Case A, if $\mathsf{P}(\gamma =\infty )>0$, then
it follows from $T_{n}\to \infty $ a.s. that $r_{n}\to \infty $. In remaining
cases where $\mathsf{P}(\gamma =t_{G})=0$, we obtain from
$\mathsf{P}(r_{n}<\gamma )\to 0$ that $r_{n}\to t_{G}$, $n\to \infty $. The
claim follows since $M_{t\wedge T_{n}} = M_{t\wedge r_{n}}$, and hence
$(M_{t})_{t\leqslant r_{n}}$ is a martingale.\index{martingale}
\end{proof}

\begin{proof}[Proof of Proposition~\ref{prop:solutions}]
(a) It is obvious that \eqref{eq:main} is equivalent to \eqref{eq:F}. It
also follows from \eqref{eq:main} that in Case B \eqref{eq:mainB} is equivalent
to \eqref{eq:F0}. Thus it remains to prove that $F$ defined in (a) satisfies
$F\overset{\mathrm{loc}}{\ll }G$. Since
${\overline{G}}(s)\geqslant {\overline{G}}(t)>0$ for any $s<t<t_{G}$, we
have
\begin{equation*}
\frac{1}{{\overline{G}}(t)} = \frac{1}{{\overline{G}}(0)} + \int
\limits _{(0,t]}\frac{1}{{\overline{G}}(s){\overline{G}}(s-)}\,dG(s),
\quad t<t_{G}.
\end{equation*}
On the other hand, from \eqref{eq:F}
\begin{equation*}
F(t){\overline{G}}(t) = F(0){\overline{G}}(0) - \int \limits _{(0,t]}H(s)
\,dG(s), \quad t<t_{G}.
\end{equation*}
Combining, we obtain from integration by parts that
\begin{equation*}
F(t)= F(t){\overline{G}}(t)\frac{1}{{\overline{G}}(t)} = F(0) - \int
\limits _{(0,t]}\frac{H(s)}{{\overline{G}}(s-)}\,dG(s) + \int \limits _{(0,t]}
\frac{F(s)}{{\overline{G}}(s-)}\,dG(s), \quad t<t_{G}.
\end{equation*}
This shows that $F\overset{\mathrm{loc}}{\ll }G$ in Case A. In Case B we must
show additionally that the function
$\frac{|F(s)|+|H(s)|}{{\overline{G}}(s-)}$ is $dG$-integrable over
$(0,t_{G})$. But
$1/{\overline{G}}(s-)\leqslant 1/{\mathsf{P}}(\gamma =t_{G})$,
$s\leqslant t_{G}$, and $F(s)$ is bounded on $[0,t_{G})$ in view of \eqref{eq:F}. The claim follows.

(b) It is clear that the function $H(t)$, $t\in {\mathcal{T}}$, defined
as in the statement, belongs to $L^{1}_{\mathrm{loc}}(dG)$. Integrating by parts,
we get, for $t\in [0,t_{G})$,
\begin{align*}
F(t){\overline{G}}(t) &= F(0){\overline{G}}(0)-\int \limits _{(0,t]} F(s)
\,dG(s) + \int \limits _{(0,t]} {\overline{G}}(s-)\,dF(s)
\\
& = F(0){\overline{G}}(0)-\int \limits _{(0,t]} F(s)\,dG(s) + \int
\limits _{(0,t]} {\overline{G}}(s-)\frac{dF}{dG}(s)\,dG(s)
\\
& = F(0){\overline{G}}(0)-\int \limits _{(0,t]} H(s)\,dG(s),
\end{align*}
i.e. \eqref{eq:main} holds. Therefore, Condition M is satisfied. Conversely,
let \eqref{eq:main} hold. In the proof of part (a) we deduced from \eqref{eq:F} (and, hence, from \eqref{eq:main}) that
\begin{equation*}
\frac{dF}{dG}(t) = -\frac{H(t)}{{\overline{G}}(t-)} +
\frac{F(t)}{{\overline{G}}(t-)},\quad dG\text{-a.s.},
\end{equation*}
and \eqref{eq:H} follows.
\end{proof}

\begin{proof}[Proof of Theorem~\ref{th:solutions}]
Let $M=(M_{t})_{t\in \mathbb{R}_{+}}$ be a local martingale.\index{local martingale} By Proposition~\ref{prop:fv}
(v) and Theorem \ref{th:2}, $M$ has representation \eqref{eq:mr2} and,
moreover, \eqref{eq:integrab} and \eqref{eq:main+} hold. Define the function
$H(t)$, $t\in {\mathcal{T}}$, by \eqref{eq:Hdef1}. Then, see \eqref{eq:integrabL},
\begin{equation*}
{\mathsf{E}}\bigl (|H(\gamma )|{\Eins }_{\{\gamma \leqslant t\}} \bigr )
\leqslant {\mathsf{E}}\bigl (|L|{\Eins }_{\{\gamma \leqslant t\}}
\bigr )<\infty ,
\end{equation*}
which implies $H\in L^{1}_{\mathrm{loc}}(dG)$. Putting
$L'=\bigl (L-H(\gamma )\bigr ){\Eins }_{\{\gamma <\infty \}}$, we obtain \eqref{eq:integrab2} as well. Now it follows from \eqref{eq:mr3} and the
second relation in \eqref{eq:integrab2} that
\begin{equation*}
\label{eq:main++}
{\mathsf{E}}(M_{t}) = F(t){\overline{G}}(t) + \int \limits _{[0,t]}H(s)\,dG(s),
\quad t\in {\mathcal{T}},
\end{equation*}
so \eqref{eq:main} and \eqref{eq:mainB} follow from \eqref{eq:main+}. Finally, \eqref{eq:MF} follows from Proposition~\ref{prop:solutions} (a).

Conversely, let \eqref{eq:mr3} hold true with a pair $(F,H)$ satisfying
Condition M and a random variable $L'$ satisfying \eqref{eq:integrab2}. Then, putting $L=H(\gamma )+L'$, we obtain \eqref{eq:mr2} and \eqref{eq:integrab}. It remains to note that \eqref{eq:main} and \eqref{eq:mainB} (in case B) imply \eqref{eq:main+}.
\end{proof}

\begin{proof}[Proof of Corollary~\ref{co:decomp}]
The required decomposition follows from \eqref{eq:mr3} if we put
\begin{equation*}
M'_{t}=F(t){\Eins }_{\{t<\gamma \}}+H(\gamma ){\Eins }_{\{t
\geqslant \gamma \}},\qquad M''_{t}=L'{\Eins }_{\{t\geqslant
\gamma \}}.
\end{equation*}
Let a local martingale\index{local martingale} $M$ with a representation \eqref{eq:mr3} vanish
on $\{t<\gamma \}$ and $\mathsf{E}M_{0} = 0$. Then $F(t)\equiv 0$ for
$t<t_{G}$ and
$0=\mathsf{E}M_{0} =H(0){\overline{G}}(0)+\mathsf{E}\bigl (L'{\Eins }_{\{
\gamma =0\}}\bigr )=H(0){\overline{G}}(0)$ in view of the second relation
in \eqref{eq:integrab2}. By Theorem~\ref{th:solutions}, it follows from \eqref{eq:main} and \eqref{eq:mainB} that $H(t)=0$ $dG$-a.s. Now, if
$M$ is also adapted with respect to the smallest filtration\index{filtration} making
$\gamma $ a stopping time,\index{stopping time} then
$M_{\gamma }= \bigl (H(\gamma )+L'\bigr ){\Eins }_{\{\gamma <\infty
\}} =L'{\Eins }_{\{\gamma <\infty \}}$ is $\sigma \{\gamma \}$-measurable.
Using again the second relation in \eqref{eq:integrab2}, we conclude that
$L'{\Eins }_{\{\gamma <\infty \}}=0$ a.s. This proves the unicity.
\end{proof}

\begin{proof}[Proof of Theorem~\ref{th:sigmamart}]
To prove sufficiency it is enough to consider the case, where
$H\equiv 0$ and $F\equiv 0$. In view of the first condition in \eqref{eq:integrabs}, there exists a Borel function
$J\colon (0,\infty ]\to \mathbb{R}_{+}$ such that
$\mathsf{E}\bigl [|L'| \big |\gamma =t\bigr ]=J(t)$. Put
$\Sigma _{n}=\Omega \times \{t\in (0,+\infty )\colon J(t)\leqslant n
\}$ and consider the Lebesgue--Stieltjes integral process
${\Eins }_{\Sigma _{n}}\cdot M_{t} = L'{\Eins }_{\bigl \{\mathsf{E}
\bigl [|L'| \big |\gamma \bigr ]\leqslant n\bigr \}}{\Eins }_{\{t
\geqslant \gamma >0\}}$. By Theorem~\ref{th:solutions}, cf.\ condition \eqref{eq:integrab2}, it is a local martingale.\index{local martingale} Since $\Sigma _{n}$ are
predictable and
$\cup _{n} \Sigma _{n}=\Omega \times \mathbb{R}_{+}$, $M$ is a
$\sigma $-martingale.

Conversely, let $M$ be a $\sigma $-martingale. It is easy to check that
to prove necessity it is enough to consider the case $M_{0}=0$. According
to Proposition~\ref{prop:fv} (v) and Remark~\ref{re:Ffv}
%
%e25 #&#
\begin{equation}
\label{eq:mrs}
M_{t}=F(t){\Eins }_{\{t<\gamma \}}+L{\Eins }_{\{t\geqslant
\gamma \}},
\end{equation}
where $L$ is a random variable, $F(t)$, $0\leqslant t<t_{G}$, is a deterministic
function with finite variation\index{finite variation} over $[0,t]$ for every $t<t_{G}$ in case
A and over $[0,t_{G})$ in Case B. By the definition of $\sigma $-martingales,
there is an increasing sequence of predictable sets\index{predictable sets} $\Sigma _{n}$ such
that $\cup _{n} \Sigma _{n}=\Omega \times \mathbb{R}_{+}$ and the integral
process ${\Eins }_{\Sigma _{n}}\cdot M$ is a local martingale\index{local martingale} for every
$n$. It was mentioned in the proof of Proposition~\ref{prop:fv} that the
integral is understood as the Lebesgue--Stieltjes integral. By Proposition~\ref{prop:fv}
(ii), there are Borel subsets $D_{n}$ of $\mathbb{R}_{+}$ such that
${\Eins }_{\Sigma _{n}}(\omega ,t){\Eins }_{\{\gamma (\omega )
\geqslant t\}} = {\Eins }_{D_{n}}(t){\Eins }_{\{\gamma (
\omega )\geqslant t\}}$, in particular,
%
%e26 #&#
\begin{equation}
\label{eq:Dn}
\cup _{n} D_{n}\supseteq {\mathcal{T}}.
\end{equation}
According to Theorem \ref{th:solutions} and Proposition~\ref{prop:solutions},
${\Eins }_{\Sigma _{n}}\cdot M$ has a representation
\begin{equation*}
{\Eins }_{\Sigma _{n}}\cdot M_{t}= F^{n}(t){\Eins }_{\{t<
\gamma \}}+\bigl (H^{n}(\gamma )+L^{n}\bigr ){\Eins }_{\{t
\geqslant \gamma \}}
\end{equation*}
where
$\mathsf{E}\bigl (|L^{n}| {\Eins }_{\{\gamma \leqslant t\}}\bigr )<
\infty $, $t\in {\mathcal{T}}$, $\mathsf{E}[L^{n}|\gamma ]=0$,
$H^{n}\in L^{1}_{\mathrm{loc}}(dG)$, $F^{n}\overset{\mathrm{loc}}{\ll }G$,
%
%e27 #&#
\begin{equation}
\label{eq:Hn0}
H^{n}(t)=F^{n}(t)-{\overline{G}}(t-)\frac{dF^{n}}{dG}(t) = F^{n}(t-)-{
\overline{G}}(t)\frac{dF^{n}}{dG}(t),
\end{equation}
$0< t<t_{G}$, and in Case B
$H^{n}(t_{G}):=\lim _{t\upuparrows t_{G}} F^{n}(t)$.

Combining with \eqref{eq:mrs}, we get
%
%e28 #&#
\begin{equation}
\label{eq:Fn}
F^{n}(t)=\int _{(0,t]}{\Eins }_{D_{n}}(s)\,dF(s),\quad 0<t<t_{G},
\end{equation}
and
%
%e29 #&#
\begin{equation}
\label{eq:Hn}
H^{n}(\gamma )+L^{n}=\int _{(0,\gamma )}{\Eins }_{D_{n}}(s)\,dF(s)+{
\Eins }_{D_{n}}(\gamma )\bigl (L-F(\gamma -)\bigr )\quad
\text{a.s.}
\end{equation}
Since $F^{n}\overset{\mathrm{loc}}{\ll }G$, it follows from \eqref{eq:Fn} and \eqref{eq:Dn} that $F\overset{\mathrm{loc}}{\ll }G$. Substituting \eqref{eq:Fn} in \eqref{eq:Hn} and taking conditional expectation given
$\gamma $, we get
\begin{equation*}
H^{n}(\gamma )-F^{n}(\gamma -) = {\Eins }_{D_{n}}(\gamma )\bigl (H(
\gamma )-F(\gamma -)\bigr )\quad \text{a.s.,}
\end{equation*}
i.e.
%
%e30 #&#
\begin{equation}
\label{eq:HnH}
H^{n}(t)-F^{n}(t-)={\Eins }_{D_{n}}(t)\bigl (H(t)-F(t-)\bigr )
\quad dG(t)\text{-a.s.}
\end{equation}
It follows from \eqref{eq:Fn} and \eqref{eq:Hn0} that
\begin{equation*}
-{\overline{G}}(t){\Eins }_{D_{n}}(t) \frac{dF}{dG}(t) = -{
\overline{G}}(t)\frac{dF^{n}}{dG}(t)= {\Eins }_{D_{n}}(t)\bigl (H(t)-F(t-)
\bigr ) \quad dG(t)\text{-a.s.},
\end{equation*}
so, taking \eqref{eq:Dn} into account, we obtain
\begin{equation*}
H(t)= F(t)-{\overline{G}}(t-)\frac{dF}{dG}(t) \quad dG(t)\text{-a.s.}
\end{equation*}
Additionally, in Case B, the left-hand side of \eqref{eq:HnH} at
$t=t_{G}$ vanishes, hence,
$H(t_{G}):=\lim _{t\upuparrows t_{G}} F(t)$. It remains to put
$L'=L-H(\gamma )$.
\end{proof}

\begin{proof}[Proof of Theorem~\ref{th:class}]
In Case B
\begin{equation*}
\Var (M)_{\infty }\leqslant 2\Var (F)_{t_{G}-}+|L|,
\end{equation*}
and the first term is finite by Remark \ref{re:Ffv}, while
$\mathsf{E}|L|<\infty $ due to \eqref{eq:integrab}. Thus, we proceed to Case
A.

(i) Note that
\begin{equation*}
\int \limits _{[0,t_{G})} |H(s)|\,dG(s) = \mathsf{E}\bigl (|H(\gamma )|{
\Eins }_{\{\gamma <t_{G}\}}\bigr )=\mathsf{E}\bigl (|H(\gamma )|{
\Eins }_{\{\gamma <\infty \}}\bigr )
\end{equation*}
and
$\mathsf{E}\bigl (|L|{\Eins }_{\{\gamma <\infty \}}\bigr )<\infty $ if and
only if
\begin{equation*}
{\mathsf{E}}\bigl (|H(\gamma )|{\Eins }_{\{\gamma <\infty \}}\bigr )<
\infty \qquad \text{and}\qquad \mathsf{E}\bigl (|L'|{\Eins }_{\{\gamma <
\infty \}}\bigr )<\infty .
\end{equation*}
Next, if $M_{\infty }$ is well defined, then
\begin{equation*}
M_{\infty }=L{\Eins }_{\{\gamma <\infty \}}+\lim _{t\to \infty }F(t)
{\Eins }_{\{\gamma =\infty \}}.
\end{equation*}
Finally, if $\mathsf{P}(\gamma =\infty )>0$, then it follows from \eqref{eq:main} that $\lim _{t\to \infty }F(t)$ exists and is finite if
$\int \limits _{[0,t_{G})} |H(s)|\,dG(s) <\infty $. Now, combining all
above, we arrive at~(i).

(ii) If $\mathsf{P}(\gamma =\infty )>0$, then
\begin{equation*}
\Var (M)_{\infty }\leqslant 2\Var (F)_{\infty }+|L|{\Eins }_{\{
\gamma <\infty \}},
\end{equation*}
and the last term on the right has finite expectation by assumptions. Since
${\overline{G}}(t)\geqslant {\mathsf{P}}(\gamma =\infty )>0$ in the case under
consideration, it follows from assumptions and \eqref{eq:main} that
$F$ has a finite variation\index{finite variation} over $\mathbb{R}_{+}$.

From now on we assume that
$\mathsf{E}\bigl (|L'|{\Eins }_{\{\gamma <\infty \}}\bigr )<\infty $,
$\int _{[0,t_{G})} |H(s)|\,dG(s)<\infty $, and
$\mathsf{P}(\gamma =t_{G})=0$. Then $M$ is a martingale\index{martingale} on $[0,t_{G})$ by
Theorem~\ref{th:2} and it coincides with $M_{\infty }=L$ for
$t\geqslant t_{G}$. Hence, it is a (necessarily closed) submartingale\index{submartingale} (resp.\ supermartingale\index{supermartingale})
if and only if $\mathsf{E}[L-M_{t}|{\mathscr{F}}_{t}]\geqslant 0$ (resp.\ $
\leqslant 0$), for $t<t_{G}$. As in the proof of Theorem~\ref{th:2},
\begin{equation*}
L-M_{t} = 0 \quad \text{on}\ \{t\geqslant \gamma \},
\end{equation*}
hence,
\begin{equation*}
{\mathsf{E}}[L-M_{t}|{\mathscr{F}}_{t}] = \mathrm{const}{\Eins }_{\{t<
\gamma \}}.
\end{equation*}
Taking expectations, we see that this constant has the same sign as
$\mathsf{E}(L-M_{t}) = \mathsf{E}(L-M_{0})$. However,
\begin{equation*}
{\mathsf{E}}(L-M_{0}) =\mathsf{E}\bigl (H(\gamma )-M_{0}\bigr )= \int \limits _{(0,t_{G})}
H(s)\,dG(s) - F(0){\overline{G}}(0) = -\lim _{t\upuparrows t_{G}} F(t){
\overline{G}}(t),
\end{equation*}
and (iii.i) follows.

The same proof shows that if
$\int _{(0,t_{G})} H(s)\,dG(s) = F(0){\overline{G}}(0)$, then
$M$ is a uniformly integrable martingale. \xch{Therefore,}{Therefore.} to prove (iii.ii)
and (iii.iii) it is enough to show that
$\mathsf{E}(\sup _{t} |M_{t}|)<\infty $ implies \eqref{eq:H1}, and that \eqref{eq:H1} implies $\mathsf{E}\bigl (\Var (M)_{\infty }\bigr ) <\infty $.

If $M$ is a local martingale\index{local martingale} with
$\mathsf{E}(\sup _{t} |M_{t}|)<\infty $, then
$\mathsf{E}\bigl (|\Delta M_{\gamma }|{\Eins }_{\{\gamma <\infty \}}
\bigr )<\infty $. But
$|\Delta M_{\gamma }|{\Eins }_{\{\gamma <\infty \}}=|L-F(\gamma -)|{
\Eins }_{\{\gamma <\infty \}}$, hence, taking conditional expectation
given $\gamma $, we get
\begin{equation*}
\int \limits _{[0,t_{G})} |H(s) - F(s-)|\,dG(s)<\infty .
\end{equation*}
In view of \eqref{eq:H++} which is equivalent to \eqref{eq:H}, we obtain \eqref{eq:H1}.

Conversely, let \eqref{eq:H1} hold. Then
\begin{align*}
&\Var (M)_{\infty
}\\
&\quad = |L|{\Eins }_{\{\gamma =0\}}+|F(0)|{\Eins }_{\{
\gamma >0\}} + \int \limits _{(0,\gamma )}\Bigl |\frac{dF}{dG}(s)
\Bigr |\,dG(s) + |L-F(\gamma -)|{\Eins }_{\{0<\gamma <\infty \}}
\\
&\quad \leqslant 2|L|{\Eins }_{\{\gamma <\infty \}}+2|F(0)|{
\Eins }_{\{\gamma >0\}} +2\int \limits _{(0,\gamma )}\Bigl |
\frac{dF}{dG}(s)\Bigr |\,dG(s)
\end{align*}
and
\begin{align*}
{\mathsf{E}}\Bigl (\int _{(0,\gamma )}\Bigl |\frac{dF}{dG}(s)\Bigr |\,dG(s)
\Bigr ) &=\int \limits _{[0,t_{G})} \int \limits _{(0,u)}\Bigl |
\frac{dF}{dG}(s)\Bigr |\,dG(s)\,dG(u)
\\
&= \int \limits _{[0,t_{G})} \Bigl |\frac{dF}{dG}(s)\Bigr |\int
\limits _{(s,t_{G})}\,dG(u)\,dG(s)
\\
&= \int \limits _{[0,t_{G})} {\overline{G}}(s)\Bigl |\frac{dF}{dG}(s)
\Bigr |\,dG(s)<\infty .\qedhere
\end{align*}
\end{proof}

%r8 #&#
\begin{remark}
It follows from the last equalities in the proof that, due to \eqref{eq:MF} and \eqref{eq:H++} respectively, \eqref{eq:H1} implies that
the following integrals are also finite:
\begin{equation*}
\int \limits _{[0,t_{G})}|F(s-)|\,dG(s)<\infty \qquad \text{and}\qquad
\int \limits _{[0,t_{G})}|H(s)|\,dG(s)<\infty .
\end{equation*}
However, it may happen that
\begin{equation*}
\int \limits _{[0,t_{G})}|F(s)|\,dG(s)=\infty ,
\end{equation*}
see an example in \cite[Remark 3.11]{HerdegenHerrmann:16}.
\end{remark}

%s4 #&#
\section{Complements}
\label{sec:5}

%s4.1 #&#
\subsection{Single jump processes\index{single jump ! process} and their compensators\index{compensator}}
\label{subsec:5.2}

Let us consider the same setting as in Section~\ref{sec:2} and let
$V$ be a finite random variable. For simplicity, we assume that
$\{\gamma =0\}\subseteq \{V=0\}$. Then
\begin{equation*}
X_{t} = V{\Eins }_{\{t\geqslant \gamma \}}
\end{equation*}
is an adapted process of finite variation\index{finite variation} on compact intervals.

%l1 #&#
\begin{lemma}%
\label{l:locint}
The process $X=(X_{t})_{\mathbb{R}_{+}}$ is of locally integrable variation
if and only if
%
%e31 #&#
\begin{equation}
\label{eq:Vintegrab}
{\mathsf{E}}\bigl (|V| {\Eins }_{\{\gamma \leqslant t\}}\bigr )<\infty ,
\quad t\in {\mathcal{T}}.
\end{equation}
\end{lemma}

\begin{proof}
Let \eqref{eq:Vintegrab} hold. If $t_{G}\in {\mathcal{T}}$, then
$\mathsf{E}\bigl (|V| {\Eins }_{\{\gamma \leqslant t_{G}\}}\bigr )<
\infty $ means that the process $X$ itself has integrable variation. In
Case A, put
\begin{equation*}
T_{n} = \left \{
\begin{array}{ll}
t_{n}, & \text{if $\gamma >t_{n}$;}
\\
+\infty , & \text{otherwise.}
\end{array}
\right .
\end{equation*}
where $t_{n}\upuparrows t_{G}$. Then $T_{n}\uparrow \infty $ a.s.\ and
$\Var (X^{T_{n}})_{\infty }= |V|{\Eins }_{\{\gamma \leqslant T_{n}
\}}= |V|{\Eins }_{\{\gamma \leqslant t_{n}\}}$.

Conversely, let $\{T_{n}\}$ be a localizing sequence of stopping times\index{stopping time}
such that\break
$\mathsf{E}\bigl (|V|{\Eins }_{\{\gamma \leqslant T_{n}\}}\bigr )<
\infty $. If $\mathsf{P}(\gamma \leqslant T_{n})=1$ for $n$ large enough,
then $V$ is integrable. So assume that $\mathsf{P}(\gamma > T_{n})>0$ for
every $n$. By Proposition \ref{prop:1} (iv), there are numbers
$r_{n}$ such that
$\{T_{n}<\gamma \}=\{T_{n}=r_{n}<\gamma \}=\{r_{n}<\gamma \}$. Thus, we
have a sequence $r_{n}$ such that
$\mathsf{E}\bigl (|V|{\Eins }_{\{\gamma \leqslant r_{n}\}}\bigr )<
\infty $. Since $T_{n}\to \infty $ a.s.\ and the sequence
$\{T_{n}\}$ is increasing, in Case A it follows that
$r_{n}\uparrow t_{G}$, and in Case B we come to a contradiction by repeating
the arguments in the concluding part of the proof of Theorem~\ref{th:2}.
\end{proof}

From now on we will assume that $X$ is a process of locally integrable
variation, i.e. \eqref{eq:Vintegrab} holds. Our aim is to find its compensator.\index{compensator}
We can introduce a function $K$ similarly as the function $H$ is introduced
in \eqref{eq:Hdef1}:
%
%e32 #&#
\begin{equation}
\label{eq:Kdef1}
K(t)=\mathsf{E}[V|\gamma =t],\quad t\in {\mathcal{T}}.
\end{equation}
It is clear that $K\in L^{1}_{\mathrm{loc}}(dG)$ and $K(0)=0$ if
$\mathsf{P}(\gamma =0)>0$. Now define
%
%e33 #&#
\begin{equation}
\label{eq:Fdef1}
F(t) = \int \limits _{(0,t]} {\overline{G}}(s-)^{-1}K(s)\,dG(s),
\quad 0\leqslant t<t_{G},
\end{equation}
in particular, $F(0)=0$. It follows that, in Case B, the function
$F$ has a bounded variation on $[0,t_{G})$ and has a finite limit as
$t\upuparrows t_{G}$, so we put
%
%e34 #&#
\begin{equation}
\label{eq:FtG}
F(t_{G})=\lim _{t\upuparrows t_{G}}F(t).
\end{equation}

The next theorem takes its origin in \cite{Dellacherie:70}, where the case when
$V=1$, $\gamma $ is finite and $t_{G}=+\infty $ is considered.

%t5 #&#
\begin{thm}%
\label{th:incr}
Let $V$ be a random variable satisfying \eqref{eq:Vintegrab},
$K$ and $F$ defined in \eqref{eq:Kdef1}--\eqref{eq:FtG}. Then the compensator\index{compensator}
$A_{t}$ of the process
$X_{t} = V{\Eins }_{\{t\geqslant \gamma \}}$ is given by
\begin{equation*}
\label{eq:comp1}
A_{t} = F(t\wedge \gamma ) \quad \text{in Case A}
\end{equation*}
and
\begin{equation*}
\label{eq:comp2}
A_{t} = F(t\wedge \gamma ) + K(t_{G}){\Eins }_{\{\gamma
\geqslant t_{G}\}}{\Eins }_{\{t\geqslant t_{G}\}}\quad
\text{in Case B}.
\end{equation*}
\end{thm}

\begin{proof}
The process $t\wedge \gamma $ is adapted and continuous, hence, it is predictable.
It follows that $F(t\wedge \gamma )$ is predictable. Next, in Case B, the
set
$\{(\omega ,t)\colon \gamma (\omega )\geqslant t_{G},\> t\geqslant t_{G}
\}$ coincides with the intersection of predictable sets\index{predictable sets}
\begin{equation*}
\{\gamma \geqslant t_{G}\}\times [t_{G},\infty ) = \bigcap _{n}
\Bigl [\{\gamma > t_{G}-n^{-1}\}\times (t_{G}-n^{-1},\infty )\Bigr ],
\end{equation*}
therefore, $A$ is predictable. Hence it is enough to show that
$M=A-X$ is a local martingale.\index{local martingale}

We use Theorem \ref{th:solutions} and Proposition \ref{prop:solutions} (b). $M$ has the representation \eqref{eq:mr2} with
the same function $F$ and
$L=F(\gamma ){\Eins }_{\{\gamma <\infty \}}-V{\Eins }_{\{
\gamma <\infty \}} +K(t_{G}){\Eins }_{\{\gamma = t_{G}<\infty \}}$.
Define the function $H$ as in Proposition \ref{prop:solutions} (b). Then
it follows from \eqref{eq:Fdef1} that $H(t)=F(t)-K(t)$, $0<t<t_{G}$, and,
in Case B, $H(t_{G})=F(t_{G})$. On the other hand, we have
$\mathsf{E}[L|\gamma =t] = F(t)-K(t)=H(t)$, $0<t<t_{G}$, and, in Case B,
$\mathsf{E}[L|\gamma =t_{G}] = F(t_{G})-K(t_{G})+K(t_{G})=F(t_{G})=H(t_{G})$.
The claim follows.
\end{proof}

%s4.2 #&#
\subsection{Example: submartingales\index{submartingale} of class $(\Sigma )$}
\label{subsec:5.1}

Recall, see \cite{Nikeghbali:06}, that a nonnegative submartingale
$X=(X_{t})_{t\in \mathbb{R}_{+}}$ is called a submartingale\index{submartingale} of class
$(\Sigma )$ if $X_{0}=0$ and it can be decomposed as
$X_{t}=N_{t}+A_{t}$, where $N=(N_{t})_{t\in \mathbb{R}_{+}}$,
$N_{0}=0$, is a local martingale,\index{local martingale} $A=(A_{t})_{t\in \mathbb{R}_{+}}$,
$A_{0}=0$, is a \emph{continuous\/} increasing process, and the measure
$(dA_{t})$ is carried by the set $\{t \colon X_{t} = 0\}$. A typical example
is a process $X_{t}= \overline{L}_{t} - L_{t}$ which is the difference between
the running maximum $\overline{L}_{t}$ of a continuous local martingale\index{local martingale}
$(L_{t})$ and $L_{t}$ itself.

Let $X=(X_{t})_{t\in \mathbb{R}_{+}}$ be a nonnegative submartingale with
the Doob--Meyer decomposition $X_{t}=N_{t}+A_{t}$, where
$N_{0}=A_{0}=0$, $N$ is a local martingale,\index{local martingale} $A$ is a predictable increasing
process. Assume that $A_{\infty }<\infty $ a.s.\ and put
$C_{t}=\inf \{s\geqslant 0\colon A_{s}>t\}$. Then, see
\cite[Lemma 3.1]{Gushchin:18}, $X$ is of class $(\Sigma )$ if and only
if a.s.
\begin{equation*}
A_{C_{t}}=A_{\infty }\wedge t\qquad \text{and}\qquad X_{C_{t}}=X_{\infty }{
\Eins }_{\{t\geqslant A_{\infty }\}},
\end{equation*}
where a finite limit $X_{\infty }:=\lim _{t\to \infty }X_{t}$ exists a.s.\ by
\cite[Proposition 3.1]{Gushchin:18}. Therefore, the process
$M_{t}=-N_{C_{t}}$ has the representation
\begin{equation*}
M_{t}=t\wedge \gamma -V{\Eins }_{\{t\geqslant \gamma \}}, \quad
\text{where}\ \gamma =A_{\infty }\ \text{and}\ V=X_{\infty }.
\end{equation*}
$M$ may not be a local martingale.\index{local martingale} For example, take
as $L$ a Brownian motion
stopped when it hits $1$ and define $X=\overline{L}-L$, then
$M_{t}=t\wedge 1$. However, if $X$ is a submartingale\index{submartingale} of class $(D)$ then
$N$ is a uniformly integrable martingale and $M$ is also a uniformly integrable
martingale (with respect to its own filtration and, by Theorem \ref{th:2}, with respect to the single jump\index{single jump ! filtration} filtration generated by
$\gamma $ on an original space). Now we can define a function $K$ according
to \eqref{eq:Kdef1} and conclude that \eqref{eq:Fdef1} is valid with
$F(t)=t$. We may interpret \eqref{eq:Fdef1} as the equation with known
$K$ and unknown $G$. This identity says that the Lebesgue measure on
$[0,t_{G})$ is absolutely continuous with respect to $dG$ but not vice
versa. However, if the function $K(t)$ does not vanish ($dG$-a.s.) then
we obtain from \eqref{eq:Fdef1} that $dG$ is equivalent to the Lebesgue
measure on ${\mathcal{T}}$, in particular, $G$ is continuous, and
\begin{equation*}
{\mathsf{P}}(\gamma >t)=\exp \Bigl (-\int \limits _{0}^{t} \frac{dt}{K(t)}
\Bigr ),\quad t<t_{G}.
\end{equation*}
This statement coincides with Theorem 4.1 in \cite{Nikeghbali:06}. If the
function $K(t)$ may vanish, analysis of equation \eqref{eq:Fdef1} with
$F(t)=t$, known $K(t)$ and unknown $G(t)$ is done in~\cite{Vallois:1994}.

A kind of a converse statement is proved in \cite{Gushchin:20}. If, say,
a martingale $M$\index{martingale} satisfies
\begin{equation*}
M_{t}=t\wedge \gamma -V{\Eins }_{\{t\geqslant \gamma \}},
\end{equation*}
where $\gamma <\infty $ and $V\geqslant 0$, then, using Monroe's theorem
\cite{Monroe:72}, we prove that there is a Brownian motion $B$ and a finite
stopping time\index{stopping time} $T$ such that, for the stopped process $L=B^{T}$, the joint
law of its terminal value $L_{\infty }$\index{terminal value} and its maximum
$\overline{L}_{\infty }$ coincides with that of $M$, that is, with the law
of $(\gamma -V,\gamma )$. In particular, this shows that a distribution
function $G$ is the law of the maximum of a uniformly integrable continuous
martingale\index{martingale} $L$ with $L_{0}=0$ if and only if, with $F(t)=t$,
$0\leqslant t<t_{G}$, we have $F\overset{\mathrm{loc}}{\ll }G$,
$\int _{[0,t_{G})} |H(s)|\,dG(s)<\infty $, where $H$ is defined by \eqref{eq:H}, and $G(t)=o(t^{-1})$,
see conditions for $M$ to have type
3 or 4 in Theorem~\ref{th:class}. This gives an alternative proof of the
main result in~\cite{Vallois:1994}.

%\begin{appendix}
%\end{appendix}

\begin{acknowledgement}[title={Acknowledgments}]
We thank three anonymous referees for careful reading of the paper and
constructive comments and suggestions for improving the presentation. A
special thanks goes to the referee who suggested Theorem \ref{th:sigmamart}
on a characterisation of $\sigma $-martingales.
\end{acknowledgement}

%\begin{funding}
%\gsponsor[id=,sponsor-id=]{}
%\gnumber[refid=]{}
%\end{funding}

%\begin{thebibliography}{}
%\end{thebibliography}

\end{document}